\documentclass[reqno]{amsart}
\usepackage{latexsym,amssymb,amsthm,amsmath}

\usepackage{amssymb}
\usepackage{amsmath}
\usepackage{xcolor}

\theoremstyle{plain}
\newtheorem{theorem}{Theorem}[section]
\newtheorem*{Theorem B}{Theorem B}
\newtheorem*{Theorem A}{Theorem A}
\newtheorem{lemma}{Lemma}[section]

\newtheorem{corollary}{Corollary}[section]

\numberwithin{equation}{section}

\theoremstyle{remark}
\newtheorem{remark}{Remark}[section]
\begin{document}

\title[Gradient estimation along super Perelman Ricci flow]
{Gradient estimation of a generalized non-linear heat type equation along Super-Perelman Ricci flow on weighted Riemannian manifolds}

\author[Y. Li, A. Abolarinwa, S. Ghosh and S. K. Hui]{Yanlin Li, Abimbola Abolarinwa, Suraj Ghosh and Shyamal Kumar Hui}


\begin{abstract}
In this article we derive gradient estimation for positive solution of the equation
\begin{equation*}
(\partial_t-\Delta_f)u = A(u)p(x,t) + B(u)q(x,t) + \mathcal{G}(u)
\end{equation*}
on a weighted Riemannian manifold evolving along the  $(k,m)$ super Perelman-Ricci flow 
\begin{equation*}
\frac{\partial g}{\partial t}(x,t)+2Ric_f^m(g)(x,t)\ge -2kg(x,t).
\end{equation*}
As an application of gradient estimation we derive a Harnack type inequality along with a Liouville type theorem.
\end{abstract}
\maketitle
 
\section{introduction}
The study of heat type equations on Riemannian manifolds is an active area of research in modern geometric analysis. Gradient estimation is a method from theory of PDE analysis that is used to derive bounds for gradient of some solution of a PDE without explicitly calculating the solution. There are various gradient estimations that gives different properties of solutions. The Harnack estimation was originated after the work of Harnack \cite{HARNACK} where he gave an inequality comparing the heat of two different points on a space. This inequality proves Harnack's theorem and H\"older regularity for weak solutions of harmonic functions. We consider $(M, g, d\nu)$ as a complete $n$-dimensional Riemannian manifold with weighted volume measure $d\nu = e^{-f} d\mu$, where $d\mu$ denotes the Riemann volume measure of $M$ and the potential function $f$ is smooth. The weighted Laplace operator is defined by
\begin{equation}
\Delta_f u = \Delta u - \langle \nabla f, \nabla u \rangle,
\end{equation}
for some smooth function $u \in C^{\infty}(M)$, where $\Delta$ is the usual Laplace-Beltrami operator. In this paper we consider the equation 
\begin{equation}\label{heat}
 (\partial_t-\Delta_f)u=A(u)p(x,t)+B(u)q(x,t)+\mathcal{G}(u)
\end{equation}
along with the $(k,m)$ super Perelman-Ricci flow defined by
\begin{equation}\label{flow}
\frac{\partial g}{\partial t}(x,t)+2Ric_f^m(g)(x,t)\ge -2kg(x,t),
\end{equation}
where $Ric_f^m$ is the Bakry-\'Emery Ricci tensor. The functions $p(x, t)$ and $q(x, t)$ are smooth functions or at least of class $C^{1}$ in $x$ and $C^0$ in $t$. $A(u)$, $B(u)$ and $\mathcal{G}(u)$ are sufficiently smooth functions.

Hamilton \cite{Hamilton-1,Hamilton} was the person who introduces the notion of Ricci flow and successfully derived Harnack estimates for Ricci flow. After the work of Li and Yau \cite{Li-Yau} this study gained a massive popularity. They introduced the idea of Li-Yau type gradient estimate and derived Harnack estimation for positive solution of Schr\"odinger type operator on complete Riemannian manifold. Later Perelman's \cite{PERELMAN} revolutionary contribution gave a new perspective to study PDE over static as well as evolving Riemannian manifolds. One of the important application of gradient estimation is Liouville type theorems, in which one can derive curvature restrictions to find whether a solution is constant or not. One can see the work of Li \cite{Li}, Sun \cite{JSUN}, Wu \cite{WU-2} and the references therein for more information. After the popularity of the studies on gradient type estimation several mathematicians derived numerous results in this area, for instance \cite{SAZAMI,WU,MOSER,MOSER-I} and the references therein for more details. Serrin \cite{SERRIN} and Moser \cite{MOSER,MOSER-I} generalized Harnack's inequality to solutions of elliptic and parabolic partial differential equations respectively. In 2018, Cao\cite{Cao} studied gradient estimation for
\begin{equation}
\label{1.4}
\left(\Delta -\partial_t\right)u(x,t)=0,\ (x, t) \in M \times [0,T],
\end{equation}
along Ricci flow, which can be obtained by setting $\mathcal{G} = 0,\ A(u) = 0 ,\ B(u)=0\text{ and } f=$ constant in equation \eqref{heat}. Thus \eqref{heat} is a generalization of (\ref{1.4}). Dung et al. \cite{Dung} derived gradient estimation for 
\begin{eqnarray}
\label{1.5}
	(\partial_t-\Delta_f)u=Au^p+Bu^q+au\log(u)+bu
\end{eqnarray}
and
\begin{eqnarray}
\label{1.6}
	(\partial_t-\Delta_f)u=Ae^{pu}+Be^{-qu}+D.
\end{eqnarray}
If we put $A(u)=Au^p$, $B(u)=Bu^q$, $p=q=1$, $\mathcal{G}(u)=a u \log(u) + bu$ in \eqref{heat}, then the equation \eqref{heat} reduces to (\ref{1.5}) and $A(u)=Ae^{pu},\ B(u)=Be^{-qu},\ p=q=1,\ \mathcal{G}(u)=D$ in \eqref{heat} we will have the equation (\ref{1.6}). 
 Recently, the second author \cite{ABOLA1}, and the second author and Taheri  \cite{ABOLA2} introduced the study of gradient estimation for the $f$-heat equation 
$$ (\partial_t-\Delta_f)u =0$$
and  logarithmically nonlinear $f$-heat equation
$$(\partial_t-\Delta_f)u =au |\log u|^\alpha, \ \ \ \alpha >1$$
on weighted manifolds with time dependent metrics and potential evolving by the  $(k,m)$ super Perelman-Ricci flow \begin{equation}
\frac{\partial g }{\partial t} + 2Ric_f^m(g)(x,t)\ge -2kg(x,t),
\end{equation} 
where $k$ is a real constant.  Soon afterward, Taheri \cite{Taheri} derived gradient estimation for a generalized weighted $\mathcal{G}$-nonlinear parabolic equation
\begin{equation}
(\partial_t-\Delta_f)u = q(x,t) u(x,t) + \mathcal{G}(u)(x,t),
\end{equation} 
coupled with $(k,m)$ super Perelman-Ricci flow and consider some special cases of nonlinearities in $\mathcal{G}(u)$. Recently this kind of works has been further investigated by Hui et al. \cite{Sujit-3,Sujit-1,Sujit-2}

Motivated by the above works in this article we derive gradient estimation for a positive solution of (\ref{heat}) along Super-Perelman Ricci flow \eqref{flow} and as application we derive a Harnack type inequality.  We obtain a bound for the quantity $\frac{|\nabla u|}{u}$ where $u$ satisfies (\ref{heat}) along the flow (\ref{flow}) on $(M^n,g,e^{-f}d\mu)$.
\section{Preliminaries}
On a complete weighted Riemannian manifold $(M, g, e^{-f}d\mu)$ the Bakry-\'Emery Ricci tensor is defined by
\begin{equation} \label{bakry}
Ric_f^m(g)=Ric(g)+\nabla^2 f-\frac{\nabla f\otimes \nabla f}{m-n},
\end{equation}
where $\nabla^2 f$ is the hessian of $f$. It is to be mentioned that the case $m=n$ occurs only the potential function $f$ is constant, i.e., $Ric_f^m(g) = Ric(g)$. And when $m \to \infty$ we have $Ric_f^{\infty}(g)\equiv Ric_f(g) = Ric(g) + \nabla^2 f$.
\begin{lemma}\label{lemma0}
For any smooth function $u \in C^{\infty}(M)$, the weighted Bochner formla is given by
\begin{equation}
\label{2.1}
\frac12 \Delta_f {|\nabla u|^2} = |\nabla^2 u|^2 + \langle \nabla u, \nabla \Delta_f u\rangle + Ric_f (\nabla u, \nabla u).
\end{equation}
\end{lemma}
Corresponding the generalized Bakry-\'Emery Ricci tensor, from (\ref{2.1}), we get the following
\begin{equation*}
\frac12 \Delta_f {|\nabla u|^2} = |\nabla^2 u|^2 + \langle \nabla u, \nabla \Delta_f u\rangle + Ric_f^m (\nabla u, \nabla u) + \frac{\langle \nabla u, \nabla f \rangle ^2}{m-n}, 
\end{equation*}
here the inner product is given by $\langle X, Y \rangle :=g(X, Y)$.
\begin{lemma} \label{lemma1}
For any smooth function $u, v \in C^\infty(M)$ we have the followings
\begin{align}
\label{eq_Del_uv}&\Delta(uv)=u\Delta v+ 2\nabla u\nabla v+v\Delta u,\\
&\Delta(\frac{u}{v})=\frac{\Delta u}{v}-\frac{2\nabla(\frac{u}{v})\nabla v}{v}-\frac{u\Delta v}{v^2}
\end{align}
\end{lemma}

For any two points $x, y \in M$ and for any $t \in (0, T]$ the quantity $d(x, y;t)$
the geodesic distance between $x$ and $y$ under the metric g(t). For any fixed $x_0 \in M$ and $R >0$ we define a compact set
\begin{equation}
Q_{R, T} = Q_{R, T}(x_0) =  \{ (x, t): d(x, x_0;t) \leq R, 0 \leq t \leq T\}
\end{equation}
We denote $a_+ = \max\{a, 0\}$ and $a_- = \max\{-a, 0\}$.
\section{Main Results}

\begin{lemma} \label{lemma3.1}
Let $u$ be a positive solution of (\ref{heat}) on $M$, along the flow (\ref{flow}) satisfying $0 <u \le \delta$ where $\delta$ is a real number. If $w=|\nabla \log(1-h)|^2$ then $w$ satisfies
\begin{align}
\nonumber  (\partial_t-\Delta_f)w =&\frac{-[\partial_t g+2Ric_f]}{(1-h)^2}-2\left|\frac{\nabla^2 h}{(1-h)}+\frac{\nabla h \otimes \nabla h}{(1-h)^2}\right|^2-2(1-h)w^2-\frac{2h \langle \nabla h,\nabla w \rangle}{(1-h)}\\
\nonumber &+\frac{2A(\delta e^h)\langle \nabla h,\nabla p \rangle}{\delta e^h(1-h)^2}+\frac{2B(\delta e^h)\langle \nabla h,\nabla q \rangle}{\delta e^h(1-h)^2} +2w\Bigg[A'(\delta e^h)p+B'(\delta e^h)q \\
\nonumber	&+\frac{h}{(1-h)}\frac{A(\delta e^h)p}{\delta e^h}+\frac{h}{(1-h)}\frac{B(\delta e^h)q}{\delta e^h}+\mathcal{G}'(\delta e^h)+\frac{h}{(1-h)}\frac{\mathcal{G} (\delta e^h)}{\delta e^h}\Bigg],
\end{align}
where $h=\log(u/\delta)$.
\end{lemma}
\begin{remark}
	It is to be noted that $(\partial_t g)(\nabla h,\nabla h)+2Ric_f(\nabla h,\nabla h)$ is written as $[\partial_t g+2Ric_f]$ for simplicity of notation.
\end{remark}
\begin{proof}[Proof of Lemma \ref{lemma3.1}]
Let $h=\log(u/\delta)$ and $w=|\nabla h|^2/(1-h)^2$. Clearly $h\le0$. Also
\begin{equation}
\label{3.1}
 (\partial_t-\Delta_f)h=|\nabla h|^2 + \delta^{-1}e^{-h}[A(\delta e^h)p+B(\delta e^h)q+\mathcal{G}(\delta e^h)].
\end{equation}
It is easily seen that
 $\displaystyle\nabla w= \frac{\nabla|\nabla h|^2}{(1-h)^2}+\frac{2|\nabla h|^2\nabla h}{(1-h)^3}$. Recalling Lemma \ref{lemma1} we find that
\begin{align}
\nonumber \Delta_f w =& \Delta_f \left(\frac{|\nabla h|^2}{(1-h)^2}\right)\\
\nonumber &=\frac{\Delta_f |\nabla h|^2}{(1-h)^2}-\frac{2}{(1-h)^2}\nabla \left(\frac{|\nabla h|^2}{(1-h)^2}\right)\nabla (1-h)^2-\frac{|\nabla h|^2}{(1-h)^4}\Delta_f (1-h)^2\\
\label{eqn3.2}&=\frac{\Delta_f |\nabla h|^2}{(1-h)^2}+\frac{4 \langle\nabla |\nabla h|^2,\nabla h \rangle}{(1-h)^3}+\frac{2|\nabla h|^2 \Delta_f h}{(1-h)^3}+\frac{6|\nabla h|^4}{(1-h)^4}.
\end{align}
Using $\partial_t w=\frac{\partial_t |\nabla h|^2}{(1-h)^2}+\frac{2|\nabla h|^2\partial_t h}{(1-h)^3}$ and $\partial_t |\nabla h|^2 = -[\partial_t g](\nabla h,\nabla h)+2 \langle \nabla h, \nabla \partial_t h \rangle$ in the above equation, we find that
\begin{equation}
\label{6}
\partial_t w=\frac{-[\partial_t g](\nabla h,\nabla h)}{(1-h)^2}+\frac{2 \langle \nabla h,\nabla \partial_t h\rangle }{(1-h)^2}+\frac{2|\nabla h|^2 \partial_t h}{(1-h)^3}.
\end{equation}
From (\ref{3.1}), we deduce
\begin{equation*}
\partial_t h= \Delta_f h+|\nabla h|^2+\delta^{-1} e^{-h}[A(\delta e^h)p+B(\delta e^h)q+\mathcal{G}(\delta e^h)].
\end{equation*}
Thus we can conclude
	\begin{align}
		\label{7}
	\nonumber	&\langle \nabla h,\nabla \partial_t h \rangle\\
	\nonumber	&= \langle \nabla h,\nabla [\Delta_f h+|\nabla h|^2+\delta^{-1} e^{-h}[A(\delta e^h)p+B(\delta e^h)q+\mathcal{G}(\delta e^h)]] \rangle \\
		\nonumber&=\langle \nabla h,\nabla \Delta_f h \rangle +\langle \nabla h,\nabla |\nabla h|^2 \rangle+\delta^{-1} e^{-h}[A(\delta e^h) \langle \nabla h,\nabla p \rangle+B(\delta e^h) \langle \nabla h,\nabla q \rangle] \\
		\nonumber&+ |\nabla h|^2\delta^{-1} e^{-h}[{p\delta e^hA'(\delta e^h)-pA(\delta e^h)}+{q\delta e^hB'(\delta e^h)-qB(\delta e^h)}\\
		&\ \ \ \ \  \ \ \ \ \ \ \ \ \ \ \ \ \ \ \ \ \ +{\delta e^h\mathcal{G}'(\delta e^h)-\mathcal{G}(\delta e^h)}].
	\end{align}
Putting \eqref{7} in \eqref{6} we infer
	\begin{align}
		\nonumber	&\partial_t w\\
		\nonumber &=\frac{-[\partial_t g](\nabla h,\nabla h)}{(1-h)^2}+\frac{2\langle \nabla h,\nabla \Delta_f h\rangle}{(1-h)^2}+\frac{2\langle \nabla h,\nabla |\nabla h|^2\rangle}{(1-h)^2}+\frac{2A(\delta e^h) \langle \nabla h,\nabla p \rangle}{\delta e^h(1-h)^2}\\
		\nonumber
		&-\frac{2p|\nabla h|^2A(\delta e^h)}{\delta e^h(1-h)^2}+\frac{2p|\nabla h|^2A'(\delta e^h)}{(1-h)^2} +\frac{2B(\delta e^h) \langle \nabla h,\nabla q \rangle}{\delta e^h(1-h)^2}-\frac{2q|\nabla h|^2B(\delta e^h)}{\delta e^h(1-h)^2}\\
		\nonumber
		&+\frac{2q|\nabla h|^2B'(\delta e^h)}{(1-h)^2}+\frac{2|\nabla h|^2\mathcal{G}'(\delta e^h)}{(1-h)^2}
		-\frac{2|\nabla h|^2\mathcal{G}(\delta e^h)}{\delta e^h(1-h)^2}+\frac{2|\nabla h|^2\Delta_f h}{(1-h)^3}\\
		\label{eqn3.5}&+\frac{2|\nabla h|^4}{(1-h)^3}+\frac{2p|\nabla h|^2A(\delta e^h)}{\delta e^h(1-h)^3}+\frac{2q|\nabla h|^2B(\delta e^h)}{\delta e^h(1-h)^3}+\frac{2|\nabla h|^2\mathcal{G}(\delta e^h)}{\delta e^h(1-h)^3}
	\end{align}
	Next we combine (\ref{eqn3.2}), (\ref{eqn3.5}) and after simplification, we infer
	\begin{align}
		\label{9}
		\nonumber	(\partial_t-\Delta_f) w =& \frac{-[\partial_t g]}{(1-h)^2}-\frac{\Delta_f |\nabla h|^2}{(1-h)^2}+\frac{2\langle \nabla h,\nabla \Delta_f h\rangle}{(1-h)^2}+\frac{2\langle \nabla h,\nabla |\nabla h|^2\rangle}{(1-h)^2}-
		\frac{4\langle \nabla h,\nabla |\nabla h|^2\rangle}{(1-h)^3}\\
		\nonumber	&-\frac{6|\nabla h|^4}{(1-h)^4}+\frac{2|\nabla h|^4}{(1-h)^3}+\frac{2A(\delta e^h) \langle \nabla h,\nabla p \rangle}{\delta e^h(1-h)^2}+\frac{2B(\delta e^h) \langle \nabla h,\nabla q \rangle}{\delta e^h(1-h)^2}\\
		\nonumber &-\frac{2p|\nabla h|^2A(\delta e^h)}{\delta e^h(1-h)^2}- \frac{2q|\nabla h|^2B(\delta e^h)}{\delta e^h(1-h)^2}+\frac{2p|\nabla h|^2A'(\delta e^h)}{(1-h)^2}+\frac{2|\nabla h|^2\mathcal{G}'(\delta e^h)}{(1-h)^2}\\
		\nonumber &+\frac{2q|\nabla h|^2B'(\delta e^h)}{(1-h)^2}-\frac{2|\nabla h|^2\mathcal{G}(\delta e^h)}{\delta e^h(1-h)^2}+\frac{2p|\nabla h|^2A(\delta e^h)}{\delta e^h(1-h)^3}+\frac{2q|\nabla h|^2B(\delta e^h)}{\delta e^h(1-h)^3}\\
		&+\frac{2|\nabla h|^2\mathcal{G}(\delta e^h)}{\delta e^h(1-h)^3}.
	\end{align}
	Applying weighted Bochner formula in \eqref{9} we deduce
	\begin{align}
		\nonumber (\partial_t-\Delta_f) w=& \frac{-[\partial_t g+2Ric_f]}{(1-h)^2}-\frac{2|\nabla^2 h|^2}{(1-h)^2}-\frac{2\langle \nabla h,\nabla |\nabla h|^2\rangle}{(1-h)^3}
		-\frac{2\langle \nabla h,\nabla |\nabla h|^2\rangle}{(1-h)^3}\\
		\nonumber&-\frac{4|\nabla h|^4}{(1-h)^4}+ \frac{2\langle \nabla h,\nabla |\nabla h|^2\rangle}{(1-h)^2}+\frac{2A(\delta e^h) \langle \nabla h,\nabla p \rangle}{\delta e^h(1-h)^2}+\frac{2B(\delta e^h) \langle \nabla h,\nabla q \rangle}{\delta e^h(1-h)^2}\\
		\nonumber&-\frac{2p|\nabla h|^2A(\delta e^h)}{\delta e^h(1-h)^2}- \frac{2q|\nabla h|^2B(\delta e^h)}{\delta e^h(1-h)^2}+\frac{2p|\nabla h|^2A'(\delta e^h)}{(1-h)^2}+\frac{2q|\nabla h|^2B'(\delta e^h)}{(1-h)^2}\\
		\nonumber&+\frac{2p|\nabla h|^2A(\delta e^h)}{\delta e^h(1-h)^3}
		+\frac{2q|\nabla h|^2B(\delta e^h)}{\delta e^h(1-h)^3}+\frac{2|\nabla h|^2\mathcal{G}'(\delta e^h)}{(1-h)^2}-\frac{2|\nabla h|^2\mathcal{G}(\delta e^h)}{\delta e^h(1-h)^2}\\
		\nonumber &+\frac{2|\nabla h|^2\mathcal{G}(\delta e^h)}{\delta e^h(1-h)^3}-\frac{2|\nabla h|^4}{(1-h)^4}+\frac{2|\nabla h|^4}{(1-h)^3}
	\end{align}
	\begin{align}\label{11}
		\nonumber	=&\frac{-[\partial_t g+2Ric_f]}{(1-h)^2}-\frac{2}{(1-h)^2}\left|\nabla \nabla h+\frac{\nabla h \otimes\nabla h}{(1-h)}\right|^2-\frac{2\langle \nabla h,\nabla |\nabla h|^2\rangle}{(1-h)^3}\\
		\nonumber & +\frac{4|\nabla h|^4}{(1-h)^3}-\frac{2|\nabla h|^4}{(1-h)^4}+ \frac{2\langle \nabla h,\nabla |\nabla h|^2\rangle}{(1-h)^2}-\frac{4|\nabla h|^4}{(1-h)^4}\\
		\nonumber&+\frac{2A(\delta e^h) \langle \nabla h,\nabla p \rangle}{\delta e^h(1-h)^2}+\frac{2B(\delta e^h) \langle \nabla h,\nabla q \rangle}{\delta e^h(1-h)^2}\\
		\nonumber&
		+\frac{2\delta^{-1} e^{-h}|\nabla h|^2}{(1-h)^2}\Big[\delta e^hA'(\delta e^h)p+\frac{h}{1-h}A(\delta e^h)p+\delta e^hB'(\delta e^h)q\\
		&\ \ \ \ \ \ \ \ \ \ \ \ \ \ \ \ \ \ \ \ \ \ \ +\frac{h}{1-h}B(\delta e^h)q+\delta e^h\mathcal{G}'(\delta e^h)+\frac{h}{1-h}\mathcal{G}(\delta e^h) \Big].
	\end{align}
Note that
\begin{equation}
	\label{eq_3.8}
	\langle (1-h)\nabla h,(1-h)\nabla w\rangle=\langle \nabla |\nabla h|^2,\nabla h\rangle+\frac{2|\nabla h|^4}{(1-h)}.
\end{equation}
Substituting $w=\frac{|\nabla h|^2}{(1-h)^2}$ and using \eqref{eq_3.8} in \eqref{11}, we have our result.

\end{proof}
\begin{lemma}
	Let u be a positive solution of (\ref{heat}) on $M$ along the flow  (\ref{flow}) satisfying $0 <u \le \delta$ and let $h=\log(u/\delta)$, where $\delta$ is a real number. If $w=|\nabla \log(1-h)|^2$ then $w$ satisfies
	\begin{align} \label{2.14}
		\nonumber &(\partial_t-\Delta_f) w\\
		\nonumber &\le -2(1-h)w^2+2kw-\frac{2h\langle \nabla h,\nabla w \rangle}{(1-h)}+\frac{2A(\delta e^h)\langle \nabla h,\nabla p \rangle}{\delta e^h(1-h)^2}+\frac{2B(\delta e^h)\langle \nabla h,\nabla q \rangle}{\delta e^h(1-h)^2}\\
		\nonumber &+2w\Bigg[A'(\delta e^h)p+\frac{h}{(1-h)}\frac{A(\delta e^h)p}{\delta e^h}+B'(\delta e^h)q+\frac{h}{(1-h)}\frac{B(\delta e^h)q}{\delta e^h}\\
		&\ \ \ \ \ \ \ \ \ +\mathcal{G}'(\delta e^h)+\frac{h}{(1-h)}\frac{\mathcal{G} (\delta e^h)}{\delta e^h}\Bigg].
	\end{align}
\end{lemma}
\begin{proof}
The proof directly follows from Lemma \ref{lemma3.1}.
\end{proof}
The following Lemma helps us to construct suitable smooth cut-off functions for the proof of our main Theorem. For detail description see \cite{Li-Yau,SOUPLET-ZHANG} and for related applications see for examples \cite{JSUN,Taheri}.
\begin{lemma}[\cite{Taheri}] \label{lemma 3.3}
For $R, T > 0$ and $\tau \in (0,T]$, there exists a smooth function $\bar{\psi} : [0,\infty)\times[0,T]\rightarrow \mathbb{R}$ such that
	\begin{enumerate}
		\item[\rm{(i)}] $\text{supp }\bar{\psi}(r,t)\subset[0,R]\times[0,T]$ and $0\le \bar{\psi}(r,t)\le 1$ in $[0,R] \times [0,T]$,
		\item[\rm{(ii)}] $\bar{\psi}=1$ in $[0,R/2]\times[\tau,T]$ and $\frac{\partial \bar{\psi}}{\partial r}=0$ in $[0,R/2]\times[0,T]$,
		\item[\rm{(iii)}] $|\frac{\partial \bar{\psi}}{\partial t}|\le \frac{c\sqrt{\bar{\psi}}}{\tau}$ on $[0,\infty) \times [0,T]$ for some $c>0$
		and $\bar{\psi}(r,0)=0\ \ \forall r \in [0,\infty)$,
		\item[\rm{(iv)}] $\frac{-c_\epsilon \bar{\psi}^\epsilon}{R}\le \frac{\partial \bar{\psi}}{\partial r}\le 0$ and $|\frac{\partial^2\bar{\psi}}{\partial^2 r}|\le \frac{c_\epsilon \bar{\psi^\epsilon}}{R^2}$ hold on  $[0,\infty)\times[0,T]$,
		for every $0<\epsilon<1$ and some $c_\epsilon>0$.
	\end{enumerate}
\end{lemma}

\begin{theorem}\label{th_main}
Let $u$ be a positive bounded solution of \eqref{heat} with $0 <u \le \delta$ in ${Q}_{R,T}$. For some $k_m=(m-1)\tilde{k},\ \tilde{h}\ge 0$, $R\ge 2$ and $T>0$ if $Ric_f^m \ge -k_m g$ and $\partial_t g \ge -2\tilde{h}g$ in $Q_{R,T}$  then there exists $C>0$ depending on $n$ and $m$ such that for every $(x,t) \in Q_{R/2,T}$ with $ t > 0$ we have
\begin{align}\label{main}
\nonumber \frac{|\nabla u|}{u}\le& C \bigg[\frac{1}{R}+\sqrt{\frac{[\gamma_{\Delta_f}]_+}{R}}+\sqrt{\frac{1}{t}}+\sqrt{\mathcal{H}}+N_p^2+N_q^2+M_A+M_B+\sqrt{M_\Gamma}\\
&\ \ \ \ +\left(\frac{|A(u)|}{u}\right)^{2/3}+\left(\frac{|B(u)|}{u}\right)^{2/3}\bigg]\bigg(1-\log(\frac{u}{\delta})\bigg),
\end{align}
where
\begin{eqnarray}
		\mathcal{H}=\sqrt{\tilde{h}^2+\tilde{k}^2}\ \ , \ \ k_m=(m-1)\tilde{k},\\
		\label{Nq}	N_q=\sup_{{Q}_{R,T}}\left[q_+^{1/2}+|\nabla q|^{1/3}\right], q_+=[q(x,t)]_+,\\
		\label{Np} N_p=\sup_{{Q}_{R,T}}\left[p_+^{1/2}+|\nabla p|^{1/3}\right], p_+=[p(x,t)]_+,\\
		\gamma_{\Delta_f}=\max_{(x,t)}\{\Delta_f r(x,t):d(x,x_0;t)=1,\ 0 \le t\le T\},\\
	\nonumber	[\gamma_{\Delta_f}]_+=max(\gamma_{\Delta_f},0),\\
		\label{M_gamma} M_\mathcal{G}(u)=\sup_{{Q}_{R,T}}\left\{\left[\mathcal{G}'(u)
		+\frac{\log(\frac{u}{D})\mathcal{G}(u)}{u(1-log(\frac{u}{D}))}\right]_+\right\},\\
		\label{M_A} M_A(u)=\sup_{{Q}_{R,T}}\left\{\left[A'(u)
		+\frac{\log(\frac{u}{D})A(u)}{u(1-log(\frac{u}{D}))}\right]_+\right\},\\
		\label{M_B} M_B(u)=\sup_{{Q}_{R,T}}\left\{\left[B'(u)
		+\frac{\log(\frac{u}{D})B(u)}{u(1-log(\frac{u}{D}))}\right]_+\right\}.
	\end{eqnarray}
When $m<\infty$, the term $\sqrt{\frac{[\gamma_{\Delta_f}]_{+}}{R}}$ in \eqref{main} can be avoided.
\end{theorem}

\begin{proof}
We now show that the estimate in the theorem	holds for all $(x, \tau)$ satisfying $d(x,x_0;\tau)\le R/2$.  As $\tau$ is arbitrary so the assertion is justified for all $(x,t)$ in $Q_{R/2,T}$ with $t>0$. Here for fixed $\tau \in (0,T] $ we consider the function $\psi(x,t)= \bar{\psi}(r(x,t),t)$, where $\bar{\psi}$ satisfies all the properties of Lemma \ref{lemma 3.3}.\\
First, we note that a direct computation using \eqref{eq_Del_uv} of Lemma \ref{lemma1} yields
\begin{equation}\label{eq_new1}
	(\partial_t-\Delta_f)(\psi w)=\psi(\partial_t-\Delta_f)w-2\langle \nabla \psi,\nabla w \rangle+w(\partial_t-\Delta_f)\psi.
\end{equation}
Similarly (by direct computation)	
\begin{equation}
	\nonumber \langle \nabla \psi,\nabla(\psi w )\rangle= \psi \langle \nabla \psi, \nabla w \rangle + w\langle \nabla \psi, \nabla \psi \rangle,
\end{equation}
which implies
\begin{equation}\label{eq_new2}
	\langle \nabla \psi, \nabla w \rangle = \frac{1}{\psi}\langle \nabla \psi, \nabla (\psi w)-w\nabla \psi \rangle.
\end{equation}
Now combining \eqref{eq_new1} and \eqref{eq_new2} we derive the evolution for the quantity $(\psi w)$ as follows
\begin{equation}\label{eq_new3}
	(\partial_t-\Delta_f)(\psi w) = \psi (\partial_t-\Delta_f)w-\frac{2}{w}\langle \nabla \psi ,\nabla (\psi w)-w\nabla \psi \rangle + w(\partial_t-\Delta_f)\psi.
\end{equation}
To begin the proof we combine \eqref{2.14} and \eqref{eq_new3} and get
	\begin{align}
		\nonumber (\partial_t-\Delta_f)(\psi w)\le& \frac{-2h\langle \psi\nabla h,\nabla w\rangle}{(1-h)}-2(1-h)\psi w^2+2k\psi w-\frac{2}{\psi}\langle \nabla\psi,\nabla(\psi w)-w\nabla\psi\rangle\\
		\nonumber & + w(\partial_t-\Delta_f)\psi +\frac{2\psi A(\delta e^h)\langle \nabla h,\nabla p \rangle}{\delta e^h(1-h)^2}+\frac{2\psi B(\delta e^h)\langle \nabla h,\nabla q \rangle}{\delta e^h(1-h)^2}\\
		\nonumber &+2\psi w\Bigg[A'(\delta e^h)p+\frac{h}{(1-h)}\frac{A(\delta e^h)p}{\delta e^h}+B'(\delta e^h)q+\frac{h}{(1-h)}\frac{B(\delta e^h)q}{\delta e^h}\\
		&\ \ \ \ \ \ \ \ \ \ \ \ +\mathcal{G}'(\delta e^h)+\frac{h}{(1-h)}\frac{\mathcal{G} (\delta e^h)}{\delta e^h}\Bigg].
	\end{align} 
\\ Since $\langle \psi \nabla h,\nabla w\rangle=\langle \nabla h,\nabla(\psi w)\rangle-w\langle \nabla h,\nabla \psi\rangle$, hence we have
	\begin{align} \label{2.17}
		\nonumber &(\partial_t-\Delta_f)(\psi w)\\
		\nonumber &\le \frac{-2h\langle \nabla h,\nabla(\psi w)\rangle}{(1-h)}+\frac{2hw\langle \nabla h,\nabla \psi\rangle}{(1-h)}-2(1-h)\psi w^2
		-\langle \frac{2\nabla\psi}{\psi},\nabla(\psi w)-w\nabla\psi\rangle\\
		\nonumber &\ \ \ +2k\psi w +w\partial_t\psi-w\Delta_f\psi+\frac{2\psi A(\delta e^h)\langle \nabla h,\nabla p \rangle}{\delta e^h(1-h)^2}+\frac{2\psi B(\delta e^h)\langle \nabla h,\nabla q \rangle}{\delta e^h(1-h)^2}\\
		\nonumber &\ \ \ +2\psi w\Bigg[A'(\delta e^h)p+\frac{h}{(1-h)}\frac{A(\delta e^h)p}{\delta e^h}+B'(\delta e^h)q+\frac{h}{(1-h)}\frac{B(\delta e^h)q}{\delta e^h}\\
		&\ \ \ \ \ \ \ \ \ \ \ \ \ +\mathcal{G}'(\delta e^h)+\frac{h}{(1-h)}\frac{\mathcal{G} (\delta e^h)}{\delta e^h}\Bigg].
	\end{align}
Let us assume that, at the point $(x_1,t_1)$, the quantity $\psi w$ attains its maximum on $Q_{R,T}$. By Calabi's \cite{CALABI} argument, without loss of generality it can be assumed that $x_1$ is not on the cut locus of $M$. We also assume that $(\psi w)(x_1,t_1) > 0$ as otherwise the desired result becomes trivial with $w(x,\tau)\le 0$ whenever $d(x,x_0;\tau)\le {R/2}$. Thus in particular $t_1>0$ by $(iii)$ of Lemma \ref{lemma 3.3} and at $(x_1,t_1)$ we have $\Delta_f(\psi w)\le 0$, $\partial_t(\psi w)\ge 0$ and $\nabla(\psi w)=0$. So From \eqref{2.17} we find that
	\begin{align} \label{2.18}
		\nonumber \psi w^2\le& \frac{hw\langle \nabla h,\nabla \psi\rangle}{(1-h)^2}+\frac{w |\nabla\psi|^2}{\psi(1-h)}+\frac{\psi A(\delta e^h)\langle \nabla h,\nabla p \rangle}{\delta e^h(1-h)^3}+\frac{\psi B(\delta e^h)\langle \nabla h,\nabla q \rangle}{\delta e^h(1-h)^3}\\
		\nonumber &+\frac{\psi w}{(1-h)}\Bigg[A'(\delta e^h)p+\frac{h}{(1-h)}\frac{A(\delta e^h)p}{\delta e^h}+B'(\delta e^h)q+\frac{h}{(1-h)}\frac{B(\delta e^h)q}{\delta e^h}\\
		&+\mathcal{G}'(\delta e^h)+\frac{h}{(1-h)}\frac{\mathcal{G} (\delta e^h)}{\delta e^h}\Bigg]+\frac{k\psi w}{(1-h)}+\frac{w\partial_t\psi}{2(1-h)}-\frac{w\Delta_f \psi}{2(1-h)}.
	\end{align}
	We now consider two cases separated according as to whether $d(x_1,x_0;t_1)\ge 1$ or $d(x_1,x_0;t_1)\le 1$ respectively.\\
	\textbf{Case 1.}
	{$d(x,x_0;t_1)\ge 1$}\\
	 In this case we obtain the upper bounds for each of the terms on the right hand side of \eqref{2.18} by appling properties (i)-(iv) of Lemma \ref{lemma 3.3}, Cauchy-Schwartz inequality and Young inequalities respectively.  For the first two terms noting that $-1<\frac{h}{1-h}\le 0$ and $\frac{1}{1-h}\le 1$ we can write
	\begin{eqnarray}\label{2.19}
		\nonumber \frac{hw\langle \nabla h,\nabla \psi\rangle}{(1-h)^2}+\frac{w |\nabla\psi|^2}{\psi(1-h)}&\le& \frac{|h||\nabla h|}{(1-h)^2}|\nabla\psi|w+\frac{\sqrt{\psi}w|\nabla\psi|^2}{(1-h)\psi^{3/2}}\\
		\nonumber &\le& \psi^{3/4} w^{3/2}\frac{|\nabla\psi|}{\psi^{3/4}}+\sqrt{\psi}w\frac{|\nabla\psi|^2}{\psi^{3/2}}
		\\
		&\le& \frac{1}{5}\psi w^2+\frac{C}{R^4}.
	\end{eqnarray}
We first estimate the terms related to $p$ using \eqref{Np} and \eqref{M_A} with Cauchy-Schwartz and Young inequality as follows
	\begin{eqnarray}
		\nonumber &&\frac{\psi A(\delta e^h)\langle \nabla h,\nabla p \rangle}{\delta e^h(1-h)^3}+\frac{p\psi wA(\delta e^h)h}{\delta e^h(1-h)^2}+\frac{p\psi w A'(\delta e^h)}{(1-h)}\\
		\nonumber &=&\frac{\psi A(\delta e^h)\langle \nabla h,\nabla p \rangle}{\delta e^h(1-h)^3}+\frac{p\psi w}{(1-h)}\bigg[A'(\delta e^h)+\frac{h}{1-h}\frac{A(\delta e^h)}{1-h}\bigg]\\
		\nonumber &\le& \frac{\psi^{1/4}|A(\delta e^h)|}{\delta e^h}\frac{\sqrt{w}|\nabla p|}{(1-h)^2}\psi^{3/4}
		+\frac{\sqrt{\psi}w p_+}{(1-h)^2}\sqrt{\psi}M_A\\
		\nonumber &\le& \frac{1}{10}\psi w^2+C\left[|\nabla p|^{8/3}+p_+^4\right]
		+\frac{1}{2}\left[\left(\frac{|A(\delta e^h)|}{\delta e^h}\right)^{8/3}+M_A^4\right]\\
		\label{eq_p}
		 &\le& \frac{1}{10}\psi w^2+CN_p^8+\frac{1}{2}\left[\left(\frac{|A(\delta e^h)|}{\delta e^h}\right)^{8/3}+M_A^4\right].
	\end{eqnarray}
	Using symmetry we deduce
	\begin{eqnarray}
		\nonumber &&\frac{\psi B(\delta e^h)\langle \nabla h,\nabla q \rangle}{\delta e^h(1-h)^3}+\frac{q\psi wB(\delta e^h)h}{\delta e^h(1-h)^2}+\frac{q\psi w B'(\delta e^h)}{(1-h)}\\
		\label{eq_q}
		 &&\le \frac{1}{10}\psi w^2+CN_q^8+\frac{1}{2}\left[\left(\frac{|B(\delta e^h)|}{\delta e^h}\right)^{8/3}+M_B^4\right].
	\end{eqnarray} 
	Combining \eqref{eq_p} and \eqref{eq_q} we infer the estimation for the terms involving $p$ and $q$ as
	\begin{align}
		\nonumber &\frac{\psi A(\delta e^h)\langle \nabla h,\nabla p \rangle}{\delta e^h(1-h)^3}+\frac{p\psi wA(\delta e^h)h}{\delta e^h(1-h)^2}+\frac{p\psi w A'(\delta e^h)}{(1-h)}+\frac{\psi B(\delta e^h)\langle \nabla h,\nabla q \rangle}{\delta e^h(1-h)^3}\\
		\nonumber&+\frac{q\psi wB(\delta e^h)h}{\delta e^h(1-h)^2}+\frac{q\psi w B'(\delta e^h)}{(1-h)}\\
		&\le \frac15 \psi w^2+C[N_p^8+N_q^8]+\frac12\left[M_A^4+M_B^4+\left(\frac{|A(\delta e^h)|}{\delta e^h}\right)^{\frac83}+\left(\frac{|B(\delta e^h)|}{\delta e^h}\right)^{\frac83}\right].
	\end{align}
Next for the terms involving $\mathcal{G}$, using \eqref{M_gamma} with Young inequality, we get
	\begin{equation}
\frac{\psi w}{(1-h)}\left[\mathcal{G}'(\delta e^h)+\frac{h}{(1-h)}\frac{\mathcal{G} (\delta e^h)}{\delta e^h}\right] \le \frac{\sqrt{\psi}w}{1-h}M_\mathcal{G} \le \sqrt{\psi} w M_\mathcal{G}
\le \frac15 \psi w^2+CM_\mathcal{G}^2.
	\end{equation}
	For the term $\frac{k\psi w}{1-h}$, using $k=\tilde{h}+(m-1)\tilde{k}$, we find the estimate as
	\begin{equation}
\frac{k\psi w}{1-h}\le \frac{\sqrt{\psi}w}{1-h}k_+\sqrt{\psi}\le \frac{\psi w^2}{15}+Ck^2.
	\end{equation}
Next for the term involving $\Delta_f\psi$ we use generalized Laplacian comparison theorem and the lower bound of $Ric_f^m(g)$. For this we consider two cases $n \le m < \infty$ and $m=\infty$. As lower bound on $Ric_f$ is a weaker assumption than a lower bound on $Ric_f^m$ we get a slightly sharper estimate in the latter case by avoiding the term $[\gamma_{\Delta_f}]_+$as conclusion.\\
	\textbf{Step 1}\ $(n \le m <\infty)$\\
By the generalized Laplacian theorem we have $\Delta_f r\le (m-1)\sqrt{\tilde{k}}\coth(\sqrt{\tilde{k}}r)$.
	As $\psi=\bar{\psi}(r(x,t),t)$, so $\Delta_f \psi=\bar{\psi}_{rr}|\nabla r|^2+\bar{\psi}_r\Delta_fr$. Noting that $\bar{\psi_r}\le 0$
\begin{equation}
\Delta_f\psi\ge\bar{\psi}_{rr}+(m-1)\bar{\psi}_r\sqrt{\tilde{k}}\coth(\sqrt{\tilde{k}}r).
\end{equation}
Now $\coth(\sqrt{\tilde{k}}r)\le \coth(\frac{\sqrt{\tilde{k}}R}{2})$ and $
	\sqrt {\tilde{k}} \coth(\frac{\sqrt{\tilde{k}}R}{2})\le (\frac{2}{R}+\sqrt{\tilde{k}})$ as $\frac{R}{2}\le r\le R$. Also $\bar{\psi}_r=0$ for $ 0\le r\le \frac{R}{2}$. Hence
	\begin{eqnarray}
		\nonumber -\Delta_f\psi&\le& -\bigg[\bar{\psi}_{rr}+(m-1)\bar{\psi}_r\sqrt{\tilde{k}}\coth(\frac{\sqrt{\tilde{k}}R}{2})\bigg]\\
		\nonumber &\le& -\bigg[\bar{\psi}_{rr}+(m-1)\bar{\psi}_r\left(\frac{2}{R}+\sqrt{\tilde{k}}\right)\bigg]\\
		&\le& |\bar{\psi}_{rr}|+(m-1)|\bar{\psi}_r|\left(\frac{2}{R}+\sqrt{\tilde{k}}\right).
	\end{eqnarray}
	Using (iv) of Lemma \ref{lemma 3.3} we see that
	\begin{eqnarray}
		\nonumber -\frac{w\Delta_f\psi}{2(1-h)}&\le& \frac{w\sqrt{\bar{\psi}}}{2(1-h)}
		\bigg[\frac{|\bar{\psi}_{rr}|}{\sqrt{\bar{\psi}}}+(m-1)\left(\frac{2}{R}+\sqrt{\tilde{k}}\right)\frac{|\bar{\psi}_r|}{\sqrt{\bar{\psi}}}\bigg]\\
		\nonumber &\le& \frac{w\sqrt{\bar{\psi}}}{2}\bigg[\frac{|\bar{\psi}_{rr}|}{\sqrt{\bar{\psi}}}+m\left(\frac{2}{R}+\sqrt{\tilde{k}}\right)\frac{|\bar{\psi}_r|}{\sqrt{\bar{\psi}}}\bigg]\\
		\nonumber &\le& \frac{1}{15}\psi w^2+C\bigg[\frac{|\bar{\psi}_{rr}|^2}{{\bar{\psi}}}+m^2\left(\frac{1}{R^2}+\tilde{k}\right)\frac{|\bar{\psi}_r|^2}{{\bar{\psi}}}\bigg]\\
		&\le& \frac{1}{15}\psi w^2+C\frac{m^2}{R^2}\left(\frac{1+\tilde{k}R^2}{R^2}\right).
	\end{eqnarray}
\textbf{Step 2}\ $(m=\infty)$\\
Note that $Ric_f(g)\ge-(n-1)\tilde{k}g$ with $ \tilde{k} \ge 0$,  $[\gamma_{\Delta_f}]_+=\max\{\gamma_{\Delta_f},0\}$ and (2.2) in Theorem 3.1 of \cite{Comparison Geometry} gives that $\Delta_fr(x,t)\le [\gamma_{\Delta_f}]_+(n-1)\tilde{k}(R-1) \text{ for }r\ge1$. As $\psi=\bar{\psi}(r(x,t),t)$ so $\Delta_f \psi=\bar{\psi}_{rr}|\nabla r|^2+\bar{\psi}_r\Delta_fr$. Also as $\bar{\psi_r}\le 0$, hence $-\Delta_f\psi\ge \bar{\psi_{rr}}+\bar{\psi_r}\{[\gamma_{\Delta_f}]_++(n-1)\tilde{k}(R-1)\}$. Therefore combining these results we have
	\begin{eqnarray}
\nonumber -\Delta_f\psi&\le& -\bar{\psi}_{rr}-\bar{\psi}_r\{[\gamma_{\Delta_f}]_++(n-1)\tilde{k}(R-1)\}\\
\label{3.31}&\le& |\bar{\psi}_{rr}|+|\bar{\psi_r}|\{[\gamma_{\Delta_f}]_++(n-1)\tilde{k}(R-1)\}.
	\end{eqnarray} 
Using the above results and Lemma \ref{lemma 3.3} with Young inequality  we can derive the bound for the term involving $\Delta_f\psi$ in \eqref{2.18} as
	\begin{eqnarray}
		\nonumber -\frac{w\Delta_f\psi}{2(1-h)}&\le& \frac{w\sqrt{\bar{\psi}}}{2(1-h)}
		\bigg[\frac{|\bar{\psi}_{rr}|}{\sqrt{\bar{\psi}}}+\frac{|\bar{\psi_r}|}{\sqrt{\bar{\psi}}}\{[\gamma_{\Delta_f}]_++(n-1)\tilde{k}(R-1)\}\bigg]\\
		\nonumber &\le& \frac{w\sqrt{\bar{\psi}}}{2}\bigg[\frac{|\bar{\psi}_{rr}|}{\sqrt{\bar{\psi}}}+\frac{|\bar{\psi_r}|}{\sqrt{\bar{\psi}}}\{[\gamma_{\Delta_f}]_++n\tilde{k}R\}\bigg]\\
		\nonumber &\le& \frac{1}{15}\psi w^2+C\bigg[\frac{{|\bar{\psi}_{rr}}^2|}{\bar{\psi}}+
		\frac{|\bar{\psi_r}|^2}{\bar{\psi}}\{[\gamma_{\Delta_f}]_+^2 +n^2\tilde{k}^2R^2\}\bigg]\\
		&\le& \frac{1}{15}\psi w^2+\frac{Cn^2}{R^2}\bigg[\frac{1}{R^2}+[\gamma_{\Delta_f}]_+{^2}+\tilde{k}^2R^2\bigg].
	\end{eqnarray}
Next we proceed to find a bound involving the term $\partial_t\psi$. For $d(x_1,x_0;t)\le R$ and fixed $t>0$, let $\phi=\phi(s):[0,d]\to M$ be a geodesic with respect to $g(t)$ connecting $x_0=\phi(0)$ and $ x_1=\phi(d)$. Using $\partial_tg\ge -2\tilde{h}g$ in $Q_{R,T}$ with $\tilde{h}\ge$ 0, we infer
\begin{eqnarray}
\nonumber \partial_tr(x_1,t)&=& \partial_t\int_{0}^{d}|\phi'(s)|_{g(t)}ds\\
\nonumber &=& \int_{0}^{d}\frac{(\partial_tg)(\phi',\phi')}{2\sqrt{g(\phi',\phi')}} ds\\
\nonumber &\ge& \int_{0}^{d}\frac{-2\tilde{h}g(\phi',\phi')}{2|\phi'|_{g(t)}}ds\\
&=& \int_{0}^{d}-\tilde{h}|\phi'|_{g(t)}ds\ge-\tilde{h}r(x_1,t)\ge-\tilde{h}R.
\end{eqnarray}
Using the properties of $\bar{\psi}$ as in Lemma \ref{lemma 3.3} we deduce
\begin{equation*}
\partial_t\bar{\psi}=\bar{\psi_t}+\bar{\psi_r}\partial_tr\le \bar{\psi_t}-\tilde{h}R\bar{\psi_r}\le \bigg[\frac{|\bar{\psi_r}|}{\sqrt{\bar{\psi}}}+R\frac{\tilde{h}|\bar{\psi_r}|}{\sqrt{\bar{\psi}}}\bigg]\le C(1+\tau \tilde{h})\frac{\psi}{\tau}
\end{equation*}
and
\begin{equation}
\frac{w\partial_t\psi}{2(1-h)}=\frac{\sqrt{\psi}w}{2(1-h)}\frac{\partial_t\psi}{\sqrt{\psi}}\le C\sqrt{\psi}w\frac{1+\tau \tilde{h}}{\tau}\le
\frac{1}{15}\psi w^2+C\frac{1+\tau^2{\tilde{h}^2}}{\tau}
\end{equation}
Thus finally we found bound for the last three terms in \eqref{2.18} subject to $Ric_f^m\ge -(m-1)\tilde{k}g$ and $k=\tilde{h}+(m-1)\tilde{k}$
\begin{align}
\nonumber &\frac{w[\Delta_f\psi+\partial_t\psi]}{2(1-h)}+\frac{k\psi w}{(1-h)}\\
\nonumber &\le \frac{1}{15}\psi w^2+C k^2+\frac{1}{15}\psi w^2+\frac{Cm^2}{R^2}\frac{1+\tilde{k}R^2}{R^2}+\frac{1}{15}\psi w^2+C\left[\frac{1}{\tau^2}+\tilde{k}^2\right]\\
\nonumber &\le\frac{\psi w^2}{5}+C\bigg[\frac{1}{R^4}+\frac{\tilde{k}}{R^2}+\frac{1}{\tau^2}+\tilde{h}^2+k^2\bigg]\\
&\le\frac{\psi w^2}{5}+C\bigg[\frac{1}{R^4}+\frac{1}{\tau^2}+\tilde{h}^2+\tilde{k}^2\bigg],
\end{align}
using $\frac{\tilde{k}}{R^2}\le \frac{\tilde{k}^2}{2}+\frac{1}{2R^4}$ with $C=C(m)>0$, and subject to $Ric_f(g)\ge -(n-1)\tilde{k}g$ with $ k=\tilde{h}+(n-1)\tilde{k}g$ the bound
\begin{eqnarray}
\frac{w[\Delta_f\psi+\partial_t\psi]}{2(1-h)}+\frac{k\psi w}{(1-h)}\le \frac{\psi w^2}{5}+C\bigg[\frac{1}{R^4}+\frac{[\gamma_{\Delta_f}]_+{^2}}{R^2}+\frac{1}{\tau^2}+\tilde{h}^2+\tilde{k}^2\bigg]
\end{eqnarray} with $C=C(n)>0$.
Using all the upper bounds of each term on right hand side of \eqref{2.18} we have
\begin{align} \label{bound}
		\nonumber \psi w^2\le& \frac15 \psi w^2+\frac{C}{R^4}+\frac15 \psi w^2+C[N_p^8+N_q^8]+\frac12\Bigg[M_A^4+M_B^4+\left(\frac{|A(\delta e^h)|}{\delta e^h}\right)^{8/3}\\
		\nonumber &+\left(\frac{|B(\delta e^h)|}{\delta e^h}\right)^{8/3}\Bigg]+\frac15 \psi w^2+CM_\mathcal{G}^2+\frac15 \psi w^2+C\bigg[\frac{1}{R^4}+\frac{[\gamma_{\Delta_f}]_+{^2}}{R^2}\\
		\nonumber &+\frac{1}{\tau^2}+\tilde{h}^2+\tilde{k}^2\bigg]\\
		\nonumber &\le \frac45 \psi w^2 +C\bigg[\frac{1}{R^4}+\frac{[\gamma_{\Delta_f}]_+{^2}}{R^2}+\frac{1}{\tau^2}
		+\mathcal{H}^2+N_p^8+N_q^8+M_A^4+M_B^4+M_\mathcal{G}^2\\
		\nonumber &+\left(\frac{|A(\delta e^h)|}{\delta e^h}\right)^{8/3}+\left(\frac{|B(\delta e^h)|}{\delta e^h}\right)^{8/3}\bigg].
\end{align}
This implies
\begin{align}
\psi w^2&\le C\bigg[\frac{1}{R^4}+\frac{[\gamma_{\Delta_f}]_+{^2}}{R^2}+\frac{1}{\tau^2}
+\mathcal{H}^2+N_p^8+N_q^8+M_A^4+M_B^4+M_\mathcal{G}^2\\
\nonumber&\ \ \ \ \ \ \ \ \ +\left(\frac{|A(\delta e^h)|}{\delta e^h}\right)^{8/3}+\left(\frac{|B(\delta e^h)|}{\delta e^h}\right)^{8/3}\bigg],
\end{align}
where $\mathcal{H}^2 = \tilde{h}^2 + \tilde{k}^2$. Using the maximal characterization of the point $(x_1,t_1)$ for any $(x,y) \in {Q_{R,T}}$, we get $[\psi w]^2(x,t)\le[\psi w]^2(x_1,t_1)\le [\psi w^2](x_1,t_1)$. Using this in \eqref{bound} we see that
\begin{eqnarray}
\nonumber [\psi w]^2&\le& C\bigg[\frac{1}{R^4}+\frac{[\gamma_{\Delta_f}]_+{^2}}{R^2}+\frac{1}{\tau^2}
		+\mathcal{H}^2+N_p^8+N_q^8+M_A^4+M_B^4+M_\mathcal{G}^2\\
&&\ \ \ \ \ \ + \left(\frac{|A(\delta e^h)|}{\delta e^h}\right)^{8/3}+\left(\frac{|B(\delta e^h)|}{\delta e^h}\right)^{8/3}\bigg].
\end{eqnarray}
Since $\psi(x,\tau)=1$ when $d(x,x_0;\tau)\le \frac{R}{2}$, hence at $(x,\tau)$
\begin{align}
\nonumber \frac{|\nabla u|}{u}\le& C \bigg[\frac{1}{R}+\sqrt{\frac{[\gamma_{\Delta_f}]_+}{R}}+\sqrt{\frac{1}{\tau}}+\sqrt{\mathcal{H}}+N_p^2+N_q^2+M_A+M_B+\sqrt{M_\mathcal{G}}\\
&\ \ \ \ \ \ +\left(\frac{|A(u)|}{u}\right)^{2/3}+\left(\frac{|B(u)|}{u}\right)^{2/3}\bigg]\bigg(1-\log(\frac{u}{\delta})\bigg).
\end{align}
\textbf{Case 2.} $d(x,x_0;t_1)\le 1$.\\
 By $(ii)$ of Lemma \ref{lemma 3.3} we have $\psi$ is constant in $d(x,x_0;t)\le \frac{R}{2},t\in[0,T]\text{ with }R\ge 2$ and so from \eqref{2.18} we find
\begin{align}
		\nonumber w^2\le&\bigg[\frac{ A(\delta e^h)\langle \nabla h,\nabla p \rangle}{\delta e^h(1-h)^3}+\frac{ B(\delta e^h)\langle \nabla h,\nabla q \rangle}{\delta e^h(1-h)^3}+\frac{pwhA(\delta e^h)}{\delta e^h(1-h)^2}+\frac{qwhB(\delta e^h)}{\delta e^h(1-h)^2}
		\\
		\nonumber &+\frac{pwA'(\delta e^h)}{(1-h)}\frac{qwB'(\delta e^h)}{(1-h)}+\frac{\psi w}{(1-h)}\left(\mathcal{G}'(\delta e^h)+\frac{h}{1-h}\frac{\mathcal{G}(\delta e^h)}{\delta e^h}\right)\\
		\nonumber &+\frac{w\partial_t\psi}{2\psi(1-h)}+\frac{kw}{(1-h)}\bigg]\\
		\nonumber w^{3/2}\le& \frac{|A(\delta e^h)||\nabla p|}{\delta e^h}+\frac{|B(\delta e^h)||\nabla q|}{\delta e^h}+\bigg(p_+M_A+q_{+}M_B+M_\mathcal{G}+\frac{C\psi^{-3/2}}{\tau}+k\bigg)\sqrt{w}
\end{align}
By using Young inequality we deduce
\begin{align}
&\nonumber w^{3/2}\\
\nonumber \le& \frac{|A(\delta e^h)||\nabla p|}{\delta e^h}+\frac{|B(\delta e^h)||\nabla q|}{\delta e^h}+\frac23\bigg(p_+M_A+q_{+}M_B+M_\mathcal{G}+\frac{C}{\tau}+k\bigg)^{\frac{3}{2}}+\frac{w^{\frac{3}{2}}}{3}\\
\nonumber \le& C\bigg[\frac{|A(\delta e^h)||\nabla p|}{\delta e^h}+\frac{|B(\delta e^h)||\nabla q|}{\delta e^h}+\bigg(p_+M_A+q_{+}M_B+M_\mathcal{G}+\frac{C}{\tau}+k\bigg)^{\frac{3}{2}}\bigg].
\end{align}
Hence
\begin{align}
\nonumber w \le&\ C\bigg[\left(\frac{|A(\delta e^h)||\nabla p|}{\delta e^h}\right)^{\frac{2}{3}}+\left(\frac{|B(\delta e^h)||\nabla q|}{\delta e^h}\right)^{\frac{2}{3}}\\
\nonumber &\ \ \ \ \ \ \ \ +\bigg(p_+M_A+q_{+}M_B+M_\mathcal{G}+\frac{C}{\tau}+k\bigg)\bigg]\\
\nonumber \le& C\bigg[\left(\frac{|A(\delta e^h)|}{\delta e^h}\right)^{4/3}+|\nabla p|^{4/3}+\left(\frac{|B(\delta e^h)|}{\delta e^h}\right)^{4/3}+|\nabla q|^{4/3}\bigg]\\
&+\bigg[C\bigg(p_+M_A+q_{+}M_B+M_\mathcal{G}+\frac{C}{\tau}+k\bigg)\bigg]
\end{align}
Now in $d(x,x_0;t)\le R/2,\ \tau \le t \le T$ we have $\psi=1$ so by $(i)$ of Lemma \ref{lemma 3.3} we have $w(x,\tau)\le \psi w(x,\tau)\le \psi w(x_1,t_1)\le w(x_1,t_1)$. So we have
	\begin{align}
		\nonumber w\le&C\bigg[\left(\frac{|A(\delta e^h)|}{\delta e^h}\right)^{4/3}+\left(\frac{|B(\delta e^h)|}{\delta e^h}\right)^{4/3}+N_p^4+N_q^4+M_A^2+M_B^2+{M_\mathcal{G}}+{\frac{1}{\tau}}+{k}\bigg]\\
		\nonumber \sqrt{w}\le& C\bigg[\left(\frac{|A(u)|}{u}\right)^{2/3}+\left(\frac{|B(u)|}{u}\right)^{2/3}+N_p^2+N_q^2+M_A\\
		&\ \ \ \ \ \ \ +M_B+\sqrt{M_\mathcal{G}}+\sqrt{\frac{1}{\tau}}+\sqrt{k}\bigg].
	\end{align}
As $\tau>0$ is arbitrary so by above we get a special case of \eqref{main}.
Thus we see that in either case the estimate holds true. This completes the proof.
\end{proof}
\begin{corollary}\label{Global}
Considering the assumptions of Theorem \ref{th_main} and  $Ric_f^m \ge -\text{k}_m g,\ \partial_t g \ge -2\tilde{h} g $ on $M \times [0,T]$, for any bounded positive solution $u$ of \eqref{heat} satisfying $0 <u \le D$ on $M$ and $0 <t <T $ we infer
\begin{eqnarray}
\nonumber \frac{|\nabla u|}{u}&\le& C \bigg[\sqrt{\frac{1}{t}}+\sqrt{\mathcal{H}}+N_p^2+N_q^2+M_A+M_B+\sqrt{M_\mathcal{G}}\\
&&\ \ \ \ \ \  + \left(\frac{|A(u)|}{u}\right)^{2/3}+\left(\frac{|B(u)|}{u}\right)^{2/3}\bigg]\bigg(1-\log(\frac{u}{\delta})\bigg)
\end{eqnarray}
Here the supremum is taken over $M \times [0,T]$ and it is assumed to be finite.
\end{corollary}
\begin{proof}
As the estimate is true for every $R \ge 2$ and as the constants are independent of $R$, by making $R \to \infty$ in \eqref{main} we get the desired result. 
\end{proof}
\section{Application}
\subsection{Parabolic Harnack Inequalities}
In this section we discuss consequences of the Theorem \ref{th_main}. Here we shall prove some Harnack inequality and a global gradient estimate with dimension free constants.
\begin{theorem}
Let $u$ be a positive bounded solution of \eqref{heat} with $0 <u \le \delta$ in $M \times [0,T]$. Let $Ric_f^m\ge -k_m g$ and $\partial_t g\ge -2\tilde{h}g$ in $M \times [0,T]$ for some $k_m,\ \tilde{h} \ge 0$, and $T>0$. For any $x_1,\ x_2\in M,t\in (0,T]$ we have
\begin{equation}
\frac{u(x_1,t)}{e\delta}\le\left[\frac{u(x_2,t)}{e\delta}\right]^\alpha,
\end{equation}
where
\begin{align*}
\alpha=\exp\biggl\{-Cd(x_1,x_2;t)\bigg[\sqrt{\frac{1}{t}}+\sqrt{\mathcal{H}}+N_p^2+N_q^2+M_A+M_B+\sqrt{M_\mathcal{G}}\\
+\left(\frac{|A(u)|}{u}\right)^{2/3}+\left(\frac{|B(u)|}{u}\right)^{2/3}\bigg]\biggr\}
	\end{align*}
\end{theorem}
\begin{proof}
	Let $x_1,\ x_2\in M$, $0 < t <T$ and also consider $\phi=\phi(s), \text{ with } 0\le s \le 1$, to be the geodesic curve joining $x_1,\ x_2$ in $M$ satisfying $\phi(0)=x_1$ and $\phi(1)=x_2$. Also let $N_p\ N_q,\ M_A,\ M_B,\ M_\mathcal{G}$ all are finite otherwise the proof will be trivial. Using Corollary \ref{Global} and noting $h=\log(\frac{u}{\delta})$ and $(1-h)>0$ we find that
\begin{eqnarray}
	\nonumber \log\bigg(\frac{1-h(x_2,t)}{1-h(x_1,t)}\bigg)&=&\int_{0}^{1}\frac{d}{ds}\log[1-h(\phi(s),t)]ds\\
	\nonumber &\le& \int_{0}^{1}-\frac{\langle \nabla h(\phi(s),t),\phi'(s)\rangle}{(1-h(\phi(s),t))}ds\\
	\nonumber &\le&\int_{0}^{1}\frac{|\nabla h|}{(1-h)}|\phi'(s)|ds.
	\end{eqnarray}
Hence
\begin{align}
\nonumber \frac{1-h(x_1,t)}{1-h(x_2,t)}\le& \exp\biggl\{-Cd(x_1,x_2;t)\bigg[\sqrt{\frac{1}{t}}+\sqrt{\mathcal{H}}+N_p^2+N_q^2+M_A+M_B\\
\nonumber &\ \ \ \ \ \ \ \ \ \ \ \ +\sqrt{M_\mathcal{G}}
+\left(\frac{|A(u)|}{u}\right)^{2/3}+\left(\frac{|B(u)|}{u}\right)^{2/3}\bigg]\biggr\}.
\end{align}
So,
\begin{align}
\nonumber \frac{\log\left[\frac{e\delta}{u(x_1,t)}\right]}{\log\left[\frac{e\delta}{u(x_2,t)}\right]}&\le
\exp\biggl\{-Cd(x_1,x_2;t)\bigg[\sqrt{\frac{1}{t}}+\sqrt{\mathcal{H}}+N_p^2+N_q^2+M_A+M_B\\
\nonumber &\ \ \ \ \ \ \ \ \ \ \ \ \ +\sqrt{M_\mathcal{G}}
+\left(\frac{|A(u)|}{u}\right)^{2/3}+\left(\frac{|B(u)|}{u}\right)^{2/3}\bigg]\biggr\}.
	\end{align}
Therefore,
\begin{equation}
\nonumber	\frac{u(x_1,t)}{e\delta}\le\left[\frac{u(x_2,t)}{e\delta}\right]^\alpha.
\end{equation}
This completes the proof.
\end{proof}

\begin{theorem}\label{3.2}
Let $u$ be a bounded positive solution to \eqref{heat} along the flow \eqref{flow} with $0 <u \le \delta$. Let $\gamma(t)$ to be smooth, non-negative but arbitrary function and consider
\begin{equation}\label{3.5}
\mathcal{F}_\gamma[u]=\frac{\gamma(t)|\nabla u|^2}{u}-u\log\left(\frac{\delta}{u}\right),
\end{equation}
then we have

\begin{align}\label{3.6}
\nonumber \bigg[\partial_t-\Delta_f-\frac{A(u)}{u}p(x,t)-\frac{B(u)}{u}q(x,t)\bigg]\mathcal{F}_\gamma[u]&\le
\frac{2\gamma}{u}\bigg[A(u)\langle \nabla u,\nabla p\rangle+B(u)\langle \nabla u,\nabla q\rangle\bigg]\\
\nonumber &+(\gamma'+2k\gamma-1)\frac{|\nabla u|^2}{u}+A(u)p\\
\nonumber & +B(u)q+\frac{2\gamma |\nabla u|^2}{u}\bigg[\mathcal{G}'(u)-\frac{\mathcal{G}(u)}{u}\bigg]\\
\nonumber &+\frac{2\gamma p|\nabla u|^2}{u}\bigg[A'(u)-\frac{A(u)}{u}\bigg]\\
&+\frac{2\gamma q|\nabla u|^2}{u}\bigg[B'(u)-\frac{B(u)}{u}\bigg]+\mathcal{F}_\gamma[u]\frac{\mathcal{G}(u)}{u},
\end{align}

\end{theorem}
\begin{proof}
Let's start by evaluating each term on the left of \eqref{3.6}.
\begin{equation*}
\partial_t\left[\frac{\gamma|\nabla u|^2}{u}\right]=\frac{\gamma'|\nabla u|^2}{u}
+\frac{\gamma}{u}\bigg[-[\partial_tg](\nabla u,\nabla u)+2\langle \nabla u,\nabla \partial_t u\rangle-\frac{|\nabla u|^2}{u}\partial_t u\bigg],
\end{equation*}
\begin{equation*}
\Delta_f\left(\frac{|\nabla u|^2}{u}\right)=\frac{\Delta_f|\nabla u|^2}{u}-\frac{2\langle \nabla |\nabla u|^2,\nabla u\rangle}{u^2}-\frac{|\nabla u|^2\Delta_f u}{u^2}+\frac{2|\nabla u|^4}{u^3},
\end{equation*}
\begin{equation*}
\partial_t(u\log(\frac{\delta}{u}))=[\log(\frac{\delta}{u}-1)]\partial_t u,
\end{equation*}	
and
\begin{equation*}
\Delta_f(u\log(\frac{\delta}{u}))=[\log(\frac{\delta}{u}-1)]\Delta_fu-\frac{|\nabla u|^2}{u}.
\end{equation*}
Thus \eqref{3.5} yields
\begin{align}
\nonumber &\bigg[\partial_t-\Delta_f-\frac{A(u)}{u}p(x,t)-\frac{B(u)}{u}q(x,t)\bigg]\mathcal{F}_\gamma[u]\\
\nonumber =& \frac{\gamma'|\nabla u|^2}{u}+\gamma\bigg(\partial_t-\Delta_f-\frac{A(u)}{u}p(x,t)-\frac{B(u)}{u}q(x,t)\bigg)\bigg(\frac{|\nabla u^2|}{u}\bigg)\\
\nonumber &-\bigg(\partial_t-\Delta_f-\frac{A(u)}{u}p(x,t)-\frac{B(u)}{u}q(x,t)\bigg)\bigg(u\log(\frac{\delta}{u})\bigg)\\
\nonumber =& \frac{\gamma}{u}\bigg[2\langle \nabla u,\nabla \partial_t u\rangle-[\partial_tg](\nabla u,\nabla u)-\Delta_f|\nabla u|^2+\frac{2\langle \nabla |\nabla u|^2,\nabla u\rangle}{u}-\frac{2|\nabla u|^4}{u^2}\bigg]\\
\nonumber &+ \frac{\gamma|\nabla u|^2}{u^2}\bigg[\partial_t-\Delta_f-\frac{A(u)}{u}p(x,t)-\frac{B(u)}{u}q(x,t)\bigg](u)-2\gamma p\frac{|\nabla u|^2}{u}\frac{A(u)}{u}-2\gamma q\frac{|\nabla u|^2}{u}\frac{B(u)}{u}\\
\nonumber &-[\log(\frac{\delta}{u}-1)]\bigg[\partial_t-\Delta_f-\frac{A(u)}{u}p(x,t)-\frac{B(u)}{u}q(x,t)\bigg](u)\\
\nonumber &+A(u)p+B(u)q+(\gamma'-1)\frac{|\nabla u|^2}{u}.
\end{align}
Using \eqref{main} and
$2\langle \nabla u,\nabla \partial_t u\rangle=2\langle \nabla u,\nabla \Delta_f u\rangle+2\langle \nabla u,\nabla [A(u)p+B(u)q+\mathcal{G}(u)]\rangle$ we see that
\begin{align}\label{3.9}
\nonumber &\bigg[\partial_t-\Delta_f-\frac{A(u)}{u}p(x,t)-\frac{B(u)}{u}q(x,t)\bigg]\mathcal{F}_\gamma[u]\\
\nonumber =& \frac{\gamma}{u}\bigg[-[\partial_tg](\nabla u,\nabla u)+2\langle \nabla u,\nabla \Delta_f u\rangle+2\langle \nabla u,\nabla [A(u)p+B(u)q+\mathcal{G}(u)]\\
\nonumber & -\Delta_f|\nabla u|^2 +\frac{2\langle \nabla |\nabla u|^2,\nabla u\rangle}{u}-\frac{2|\nabla u|^4}{u^2}\bigg]+(\gamma'-1)\frac{|\nabla u|^2}{u}+A(u)p\\
&+B(u)q -[\log(\frac{D}{u}-1)]\mathcal{G} (u)-\frac{\gamma|\nabla u|^2}{u^2}\bigg[\mathcal{G} (u)+2\gamma A(u)p+2\gamma B(u)q\bigg].
\end{align}
From the weighted Bochner formula we have
\begin{equation*}
2\langle \nabla u,\nabla \Delta_f u\rangle=\Delta_f|\nabla u|^2-2|\nabla^2 u|^2-2Ric_f(\nabla u,\nabla u).
\end{equation*}
Using this in \eqref{3.9} we infer
\begin{align}
\nonumber& \bigg[\partial_t-\Delta_f-\frac{A(u)}{u}p(x,t)-\frac{B(u)}{u}q(x,t)\bigg]\mathcal{F}_\gamma[u]\\
\nonumber =&-\frac{\gamma}{u}\bigg[[\partial_tg](\nabla u,\nabla u)+2|\nabla^2 u|^2+2Ric_f(\nabla u,\nabla u)-\frac{2\langle \nabla |\nabla u|^2,\nabla u\rangle}{u}\bigg]-\frac{2\gamma}{u}\frac{|\nabla u|^4}{u^2}\\
\nonumber &+\frac{2\gamma}{u}\langle \nabla u,\nabla(A(u)p)\rangle+\frac{2\gamma}{u}\langle \nabla u,\nabla(B(u)q)\rangle+\frac{2\gamma}{u}\langle \nabla u,\nabla(\mathcal{G}(u))\rangle+(\gamma'-1)\frac{|\nabla u|^2}{u}\\
\nonumber &-[log(\frac{\delta}{u}-1)]\mathcal{G} (u)+A(u)p+B(u)q-\frac{\gamma|\nabla u|^2}{u^2}\bigg[\mathcal{G} (u)+2\gamma A(u)p+2\gamma B(u)q\bigg].
\end{align}
Using the facts that $|\nabla^2 u|^2-\frac{\langle \nabla |\nabla u|^2,\nabla u\rangle}{u}+\frac{|\nabla u|^4}{u^2}=|\nabla^2u-\frac{\nabla u \otimes \nabla u}{u}|^2\ge 0$ and  $\partial_tg+2Ric_f\ge -2kg$ in the above equation we have our desired result.
\end{proof}

\begin{theorem}
Let $M$ be closed weighted Riemannian manifold and $u$ be a positive bounded solution of \eqref{heat} with $0<u\le \delta$ in $M \times [0,T]$ and metric-potential pair evolves under \eqref{flow} with $k\ge 0$. Also suppose that $A(u),\ B(u),\ \mathcal{G}(u)\ge 0$ and $A'(u),\ B'(u),\ \mathcal{G}'(u) \le 0$. Then
\begin{align*}
&t|\nabla \log(u)|^2 \le \\
& \frac{(1+2kt)[1+\log(\frac{\delta}{u})]+\frac{2(1+2kt)^2}{ut}(|\nabla p|^2+|\nabla q^2|^2)+\frac{(1+2kt)A(u)p}{u}+\frac{(1+2kt)B(u)q}{u}}{1-\frac{(1+2kt)^2}{ut^2}A^2(u)-\frac{(1+2kt)^2}{ut^2}B^2(u)}.
\end{align*}
\end{theorem}
\begin{proof}
Note that the function $\gamma(t)=\frac{t}{1+2kt}$ is non negative and it satisfies $\gamma'+2k\gamma-1\le 0$. Now choose $\bar{\delta}=e\delta$ in place of $\delta$ in Theorem \ref{3.2} then clearly $u \le \bar{\delta}$. Referring to \eqref{3.6} we get,
\begin{align*}
&\bigg[\partial_t-\Delta_f-\frac{A(u)}{u}p(x,t)-\frac{B(u)}{u}q(x,t)\bigg]\mathcal{F}_\gamma[u]-\frac{2\gamma}{u}\bigg[A(u)\langle \nabla u,\nabla p\rangle+B(u)\langle \nabla u,\nabla q\rangle\bigg]-A(u)p-B(u)q\\
&\le (\gamma'+2k\gamma-1)\frac{|\nabla u|^2}{u}+\frac{2\gamma |\nabla u|^2}{u}\bigg[\mathcal{G}'(u)-\frac{\mathcal{G}(u)}{u}\bigg]+\frac{2\gamma p|\nabla u|^2}{u}\bigg[A'(u)-\frac{A(u)}{u}\bigg]\\
&+\frac{2\gamma q|\nabla u|^2}{u}\bigg[B'(u)-\frac{B(u)}{u}\bigg]-log[\frac{\bar{\delta}}{u}-1]\mathcal{G}(u) \le 0.
\end{align*}
Clearly for all $x \in M$ we have, $\mathcal{F}_\gamma[u](x,0)\le 0$. Thus by an application of Maximum principle (to $\Delta_f$) we deduce, $\mathcal{F}_\gamma[u](x,t)\le 0$ in $M\times[0,T]$. From which after some simplification and applying Cauchy-Schwartz inequality and Young inequality we get the desired result.
\end{proof}
\subsection{Liouville Type Result}
Let $\mathcal{B_R}$ be the closed geodesic ball in $M$ with radius $R > 0$ and centre at $x_0$ and let $(M,g,e^{-f}d\mu)$ be a complete weighted static Riemannian manifold with $Ric_f^m \ge -k_m g$ in $\mathcal{B}\subset M$. Let $u$ be the solution of
\begin{equation}
	 \Delta_f u+A(u)p(x)+B(u)q(x)+\mathcal{G}(u)=0
\end{equation}
 

in $\mathcal{B}$ then there exists a constant $C$ such that in $\mathcal{B}_{R/2}$ we have,
\begin{align}
	\nonumber \frac{|\nabla u|}{u}\le& C \bigg[\frac{1}{R}+\sqrt{\frac{[\gamma_{\Delta_f}]_+}{R}}+\sqrt{k}+N_p^2+N_q^2+M_A+M_B+\sqrt{M_\mathcal{G}}\\
	&\ \ \ \ +\left(\frac{|A(u)|}{u}\right)^{2/3}+\left(\frac{|B(u)|}{u}\right)^{2/3}\bigg]\bigg(1-\log(\frac{u}{\delta})\bigg),
\end{align}
Letting $R\to \infty$ we get a global estimate as follows,
\begin{align}
	\nonumber \frac{|\nabla u|}{u}\le& C \bigg[\sqrt{k}+N_p^2+N_q^2+M_A+M_B+\sqrt{M_\mathcal{G}}\\
	&\ \ \ \ +\left(\frac{|A(u)|}{u}\right)^{2/3}+\left(\frac{|B(u)|}{u}\right)^{2/3}\bigg]\bigg(1-\log(\frac{u}{\delta})\bigg).
\end{align}
 If $Ric_f^m\ge 0$,$p=0,q = 0$ and $u$ is a positive bounded solution of $\Delta_fu+\mathcal{G} (u)=0$ such that $M_\mathcal{G}=0$ then we find that,
 \begin{align}
 \frac{|\nabla u|}{u}\le& C \bigg[M_A+M_B+\left(\frac{|A(u)|}{u}\right)^{2/3}+\left(\frac{|B(u)|}{u}\right)^{2/3}\bigg]\bigg(1-\log(\frac{u}{\delta})\bigg),
 \end{align}
which in particular for $A(u)=0,B(u)=0$ gives the result
	as following,
\\ If $Ric_f^m\ge 0$,$p=0,q = 0$, $A(u)=B(u)=0$ and $u$ is a positive bounded solution of $\Delta_fu+\mathcal{G} (u)=0$ such that $M_\mathcal{G}=0$  then $u$ is constant.\\
\\
\noindent{\bf Acknowledgement:} The author (S. Ghosh) gratefully acknowledges to the University Grants Commission, Government of India for the award of Junior Research Fellowship.\\ \vskip4pt
\noindent \textbf{Data availibility:} The authors confirm that the data supporting the findings of this study are available within the article.\\ \vskip4pt
\noindent \textbf{Conflict of interest:} The authors confirm that there is no conflict of interest.

\vspace{.1in}
\noindent Yanlin Li\\
\noindent School of Mathematics, Hangzhou Normal University, Hangzhou 311120, China\ E-mail: liyl@hznu.edu.cn\\

\noindent{Abimbola Abolarinwa\\ Department of Mathematics, University of Lagos, Akoka, Lagos State, Nigeria}\\
E-mail: A.Abolarinwa1@gmail.com, aabolarinwa@unilag.edu.ng\\

\noindent{Suraj Ghosh \\ Department of Mathematics, The University of Burdwan, Golapbag, Burdwan 713104, West Bengal, India}\\
Email: surajghosh0503@gmail.com\\

\noindent{Shyamal Kumar Hui \\ Department of Mathematics, The University of Burdwan, Golapbag, Burdwan 713104, West Bengal, India}\\
Email: skhui@math.buruniv.ac.in\\

\end{document}